\documentclass[12pt,letter]{amsart}

\usepackage[
    top=2.5cm, bottom=2.5cm, inner=2.5cm, outer=2.5cm,
    marginparwidth=2cm 
    ]{geometry}

\usepackage{amssymb,latexsym, amsmath, amsxtra, mathrsfs, bm}
\usepackage{graphicx}
\usepackage{amsrefs}

\usepackage{amsmath,tikz-cd}

\usepackage[utf8]{inputenc}

\usepackage{verbatim}
\usepackage[abs]{overpic}
\usepackage{hyperref}
\usepackage{mathtools}

\DeclareMathAlphabet{\mathbbold}{U}{bbold}{m}{n}

\allowdisplaybreaks[3]

\theoremstyle{plain}
        \newtheorem{theorem}{Theorem}[section]
        \newtheorem*{theorem*}{Theorem}
        \newtheorem*{conj*}{Conjecture}
        \newtheorem{lemma}[theorem]{Lemma}
        \newtheorem{prop}[theorem]{Proposition}

\theoremstyle{definition}

          \newtheorem{claim}[theorem]{Claim}

\theoremstyle{remark}

\numberwithin{equation}{section}
\numberwithin{theorem}{section}
\numberwithin{table}{section}
\numberwithin{figure}{section}

\newcommand{\R}{\mathbb{R}}

\newcommand{\N}{\mathbb{N}}

\def\({\left(}
\def\({\right)}
\def\[{\left[}
\def\]{\right]}

\newcommand{\vertiii}[1]{{\left\vert\kern-0.25ex\left\vert\kern-0.25ex\left\vert #1
    \right\vert\kern-0.25ex\right\vert\kern-0.25ex\right\vert}}

\DeclareMathOperator{\Ima}{Im}

\begin{document}
\title[Graphs of continuous functions never be self-similar]{Graphs of continuous but non-affine functions are never self-similar}
\author{Carlos Gustavo Moreira, Jinghua Xi \and Yiwei~Zhang}
\thanks{We would like to thank Lifeng Xi and Kan Jiang for posing the problem that inspired the current research. We would also like to thank Ra\'{u}l Ures, Pedro Salom\~{a}o, Tongyao Yu, as well as the anonymous referee(s) for valuable comments. C. G. Moreira was partially supported by CNPq and FRPERJ. Y. Zhang was partially supported by NSFC Nos. 12161141002, 12271432, and Guangdong Basic and Applied Basic Research Foundation No. 2024A1515010974.}





\begin{abstract}

Bandt and Kravchenko \cite{BandtKravchenko2010} proved that if a self-similar set spans $\R^m$, then there is no tangent hyperplane at any point of the set. In particular, this indicates that a smooth planar curve is self-similar if and only if it is a straight line. When restricting curves to graphs of continuous functions, we can show that the graph of a continuous function is self-similar if and only if the graph is a straight line, i.e., the underlying function is affine.

\end{abstract}

\maketitle


\section{Introduction}\label{sec:Introduction}


A map $S:\mathbb{R}^{m}\to \mathbb{R}^{m}$ is said to be a \emph{(contracting) similitude} (e.g., \cite{Falconer1985}) if $S(x)=rOx+b$, where $r\in(0,1)$, $b\in\mathbb{R}^{m}$, and $O$ is an orthogonal matrix.
A compact set $K\subset\mathbb{R}^{m}$ is \emph{self-similar} if there are similitudes $\{S_{i}\}_{i=1}^{k}$, such that
\begin{equation}\label{equ:selfsimilarset}
    K=\bigcup_{i=1}^{k}S_{i}(K).
\end{equation}

The structure of a self-similar set becomes relatively well-understood when it satisfies the open set condition (\textbf{OSC}) introduced by Hutchinson in \cite{Hutchinson1981}. Here, the OSC is satisfied if there exists a nonempty open set $V\subset \R^m$ such that
$\bigcup_{i=1}^{k}S_i(V)\subset V \text{ and } S_i(V)\cap S_j(V)=\emptyset \text{ for }i\neq j.$
In this case, $\dim_{\rm{H}}K=s$, where $dim_H$ is the Hausdorff dimension and $s$ is the unique solution of $\sum_{i=1}^{k}r_{i}^{s}=1$, with $r_i$ being the ratio of $S_i$. Another characteristic of the OSC can be found in \cite{BandtGraf2010} by Bandt and Graf.


It is widely acknowledged that fractals are inherently non-smooth. Yet, there has been limited exploration into geometric objects that exhibit both self-similarity and smoothness. However, Bandt and colleagues have been pioneers in this field, achieving significant results. For instance, Bandt and Mubarak \cite{BandtMubarak2004} established that any differentiable subcurve of the classical Sierpinski carpet must be a line segment. It is worth pointing out that Bandt and Kravchenko \cite{BandtKravchenko2010} demonstrated that a self-similar set spanning $\R^m$ cannot possess a tangent hyperplane at any point within the set, a finding with broad applications. For example, it suggests that a self-similar planar curve can only be a straight line if differentiable at some point.

In the present study, we are concerned with a special class of curves: the graphs of continuous functions, i.e., $$G\coloneqq\{(x,y)\in\R^2:y=f(x),x\in I\},$$where $f$ is a continuous function on a compact interval $I$. Many continuous but nowhere differentiable functions exhibit high ``self-similarity" in their graphs. A notable example is Takagi's function (e.g., see \cite{Teiji1901}), as illustrated in Figure \ref{fig:takagi}. However, because these functions are not smooth, it's hard to tell if they are self-similar by using existing results on smooth self-similar sets given by Bandt and Kravchenko \cite{BandtKravchenko2010}. This complexity has spurred our interest in exploring alternative approaches to address such questions.

\begin{figure}[h]
    \centering
    \includegraphics[width=0.55\textwidth]{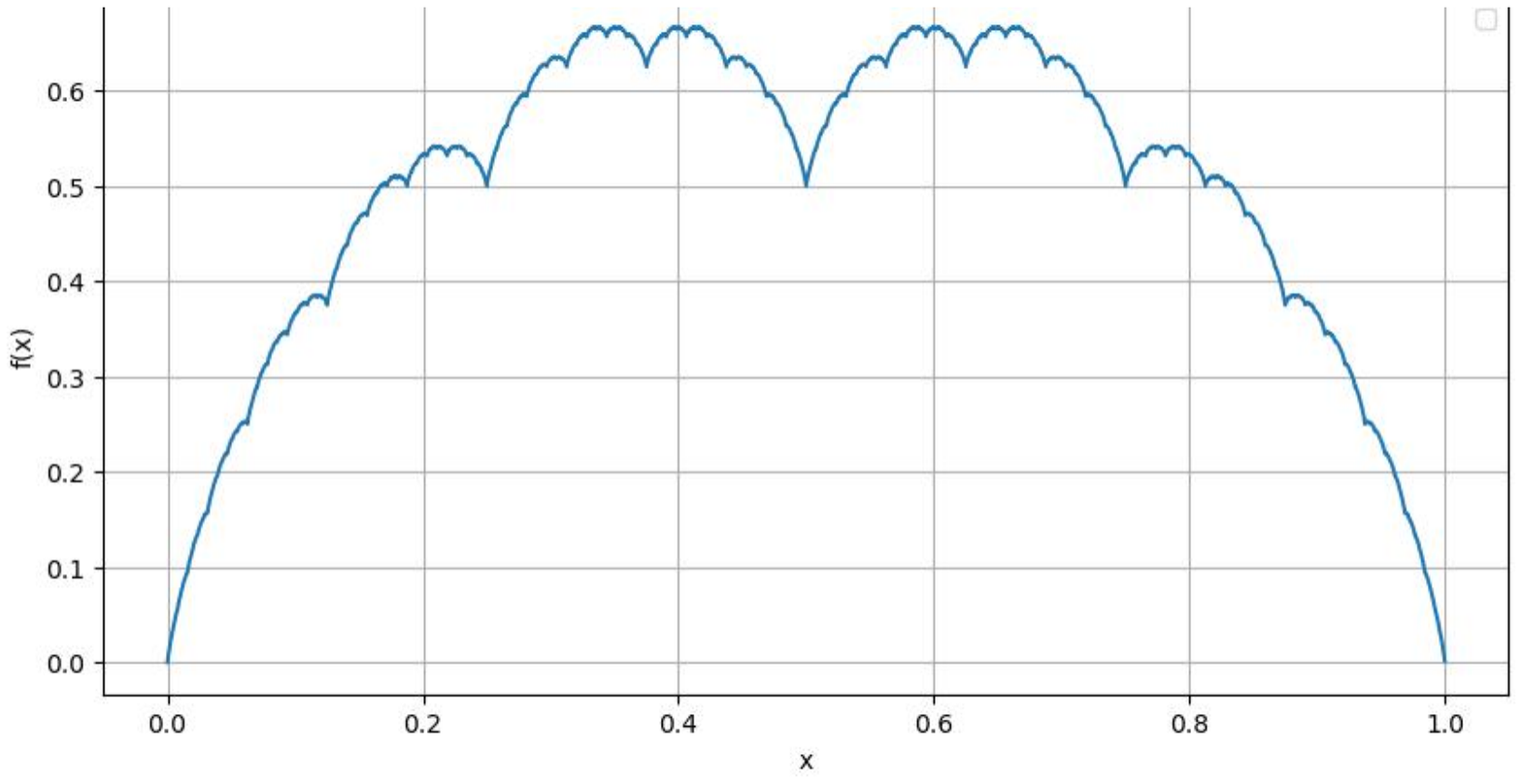}
    \caption{The graph of Takagi's function}
    \label{fig:takagi}
\end{figure}

In this work, we are concerned with which geometric shapes can be realized as a self-similar set. In particular, we propose the following problem for planar graphs associated with one real variable continuous functions:

\begin{center}
    \emph{When is the planar graph of a continuous function self-similar?}
\end{center}
This problem originates from a question proposed by the second named author and Jiang: Can the graph of a $C^{1}$ and non-affine function $f$ on a compact interval be self-similar? This is discussed in \cite[Question 3.1]{XiJiang2023}. In addition, Bandt and Kravchenko's research reveals the absence of tangent spaces in self-similar curves (see \cite[Theorem 1]{BandtKravchenko2010}), suggesting the $C^1$ regularity seems to be overly stringent for the graphs of such functions.





In this article, we are aiming to give an affirmative and comprehensive answer to the above question. To be more precise, our main result is stated as follows.

\begin{theorem}\label{thm:maintheormrigidity}
  Let $I$ be a compact interval, $f:I\to\mathbb{R}$ be a continuous function and $G=\{(x,f(x)):x\in I\}$. Then the following two statements are equivalent:
  \begin{itemize}
    \item $G$ is a self-similar set;
    \item The underlying function $f(x)=ax+b$ for some $a,b\in \R$. In other words, $f$ is an affine function on $I$.
  \end{itemize}
\end{theorem}

It is worth remarking that due to the existence of continuous self-similar planar curves, none of the hypotheses in Theorem \ref{thm:maintheormrigidity} can be weakened. Moreover, Theorem \ref{thm:maintheormrigidity} strengthens \cite[Corollary 1.6]{XiJiang2023} by the second named author and Jiang.
According to Theorem \ref{thm:maintheormrigidity}, graphs of continuous functions that are non-affine are not self-similar. In particular, the graphs of Weierstrass's function, Takagi's function, Cantor-Lebesgue's function, etc., are not self-similar.

The paper is organized as follows. In Section 2, we recall some notations and review some preliminaries. In Section 3, we prove our main Theorem \ref{thm:maintheormrigidity}. In subsection \ref{sec:3.1}, we explain the strategy of the proof. The proof consists of three steps: occupy subsections \ref{sec:slopemap}, \ref{sec:vitalicoverargument} and \ref{sec:bandtkravchenko'sresult} respectively.

\section{Notation and Preliminaries}\label{sec:preliminaries}


For $A$ is a subset of $B$, we denote by $A^c$ the complement of $A$ in $B$, assuming $B$ is evident from the context.
We use the convention that $\mathbb{N}\coloneqq\{1,2,3,\dots\}$, $\N_0\coloneqq\{0\}\cup \N$, $\mathbb{Z}$ the set of integers, and $\mathbb{Q}$ (resp. $\mathbb{Q}^{c}$) the set of all rational (resp. irrational) numbers.
For a compact interval $I$, the length of $I$ is denoted by $|I|$. For each $k\in\N$, define $[k]:=\{1,\dots,k\}$. For each integer $n\in\N$, define
\begin{equation*}\label{equ:kl}
[k]^n\coloneqq\{(x_1,\dots,x_l):x_i\in [k],\ i=1,\dots,n\}.
\end{equation*}

For a continuous function $f$ on a compact interval $I$, the \emph{graph} of $f$, denoted as $G$, is defined by $$G\coloneqq\{(x,f(x))\in \R^2:x\in I\}.$$




Next, let us introduce some notations for self-similar sets on $\R^2$.
We say a map $S:\mathbb{R}^{2}\to \mathbb{R}^{2}$ is a \emph{(contracting) similitude} if $S(x)=r\rho_\theta x+b$ with $r\in(0,1),b\in\mathbb{R}^{2}$ and $\rho_\theta=\begin{pmatrix}
    \cos\theta&\sin\theta \\ -\sin\theta &\cos\theta
\end{pmatrix}$. 
For each planar self-similar set, \eqref{equ:selfsimilarset} reduces to
\begin{equation*}\label{equ:planalsss}
G=\bigcup_{i=1}^{k}S_{i}(G),
\end{equation*}
where $S_i(v):=r_i\rho_{\theta_i}(v)+b_i$. Moreover, for $n\in\mathbb{N}$ and $\alpha=(i_1,\dots,i_n)\in[k]^n$, put
\begin{equation*}\label{equ:sralpha}
S_\alpha\coloneqq S_{i_n}\circ \cdots \circ S_{i_1}~~ \text{and}~~r_\alpha\coloneqq r_{i_n}\cdots r_{i_1}.
\end{equation*}

In addition, denote by $S^1\coloneqq\{(x,y)\in\R^2:x^2+y^2=1\}$ the unit circle centred at the origin, endowed with the circle metric. Denote by $\mathcal{N}$ and $\mathcal{S}$ the points $(0,1)$ and $(0,-1)$ on $S^1$, respectively. We say each connected open subset in $S^{1}$ an \emph{arc}.

By a well-known result on the minimality for irrational rotation on $S^1$ (see e.g., \cite[Theorem 5.8]{Walters82}), we immediately have


\begin{lemma}\label{cor:fullcircle}
    Suppose $J$ is an arc in $S^1$ and $\frac{\theta}{2\pi}\in\mathbb{Q}^c$, then
    \begin{equation}\label{equ:cover}
    \bigcup\limits_{i=0}^\infty \rho_\theta^i (J)=S^1.
    \end{equation}
\end{lemma}


\section{Proof of Theorem \ref{thm:maintheormrigidity}}\label{sec:3}
\subsection{Strategy of the proof.}\label{sec:3.1}

Let's briefly outline the strategy of the proof of Theorem \ref{thm:maintheormrigidity}.
The principal obstacle is addressing the function's lack of differentiability. Our argument steers clear of methods reliant on curvature or other differentiable mechanisms, differing significantly from \cite[Theorem 1.5]{XiJiang2023}.

Our proof can be summarized in the following three key steps:

\begin{enumerate}
    \item[Step 1:] We demonstrate that the only possible rotation $\rho_{\theta_i}$  associated with similitude $S_i(v)=r_i\rho_{\theta_i}(v)+b_i$ is either identity or reflection, i.e., $\theta_i=0$ or $\pi$.


    \item[Step 2:] We show that the underlying function $f$ of $G$ is Lipschitz continuous.


    \item[Step 3:] We prove that $f$ must be affine, i.e., $f(x)=kx+b$ for some $k,b\in \R$.

\end{enumerate}

  The proof of Theorem \ref{thm:maintheormrigidity} is divided into three subsections (Subsections \ref{sec:slopemap}, \ref{sec:vitalicoverargument}, and \ref{sec:bandtkravchenko'sresult}), each subsection corresponds to the steps outlined above. The complete proof of Theorem \ref{thm:maintheormrigidity} is provided at the end of Subsection \ref{sec:bandtkravchenko'sresult}. Without loss of generality, we reduce the hypothesis of Theorem \ref{thm:maintheormrigidity} to the interval $[0,1]$.

\subsection{Possible Rotations for Similitudes}\label{sec:slopemap}


In this subsection, we transform the dynamics of a similitude around its fixed point into the dynamics of a set of directions under rotation.

This is achieved by constructing an auxiliary function $\Phi$, which maps points on the self-similar graph $G$ to directions relative to a unique fixed point. By analyzing the behaviour of the image of $\Phi$ under the rotation $\rho_\theta$, we deduce that there are only two possible rotations that can occur within the self-similar graph $G$, as stated in Proposition \ref{prop:rigidrotation}.

\begin{prop}\label{prop:rigidrotation}
    Let $G$ be the graph of a continuous function $f$ on $[0,1]$, and $S(x)=r\rho_\theta(x)+b$ be a strict contracting similitude, with $r\in(0,1),b\in\mathbb{R}^{2}$, and $\rho_{\theta}$ is the rotation matrix. If $S(G)\subset G$, then $\rho_\theta=\begin{pmatrix}
        1&0\\0&1
    \end{pmatrix}$ or
    $\begin{pmatrix}
        -1&0\\0&-1
    \end{pmatrix}$.
\end{prop}
Since $S$ is a (strict) contraction mapping from $G$ to itself, it follows from the Banach contraction principle, similitude $S$ has a unique fixed point in $G$, which we denote by $p^*$.

Next, define a function
\begin{equation*}\label{equ:phi}
\Phi: G \setminus \{p^*\} \to S^1, \quad p \mapsto \frac{p-p^*}{\|p-p^*\|}.
\end{equation*}
Given that $G$ is the graph of a function, this ensures that the image of $\Phi$ lies in the unit circle $S^1$, excluding the points $\mathcal{N}$ and $\mathcal{S}$, i.e.,
\begin{equation}\label{completementofimage}
    \Ima(\Phi) \subset S^1 \setminus \{\mathcal{N}, \mathcal{S}\}.
\end{equation}

Recall that the points $\mathcal{N}$ and $\mathcal{S}$ were given in Section \ref{sec:preliminaries}. Moreover, by the definitions of $S$ and $\Phi$ and the hypothesis that $S(G)\subset G$, we then have
\begin{equation*}\label{equ:composition}
\Phi(S(p)) = \frac{S(p)-p^*}{\|S(p)-p^*\|} = \frac{S(p)-S(p^*)}{\|S(p)-S(p^*)\|} = \frac{\rho_\theta(p-p^*)}{\|p-p^*\|},~~\text{for every } p\in G\backslash\{p^{
*}\}.
\end{equation*}
This means $\Phi(S(p)) = \rho_\theta(\Phi(p))$. It then yields that $\Ima \Phi$ is forward invariant under the rotation $\rho_\theta$, i,e.,

\begin{equation}\label{equ:forward}
    \rho_{\theta}(\Ima\Phi)\subset \Ima\Phi.
\end{equation}

Since $G \setminus \{p^*\}$ has at most two connected components, it is worth mentioning that the continuity of $\Phi$ implies that $\Ima \Phi$ also has at most two connected components.

\begin{lemma}\label{lem:containarc}
    If $G$ is not a straight line, then $\Ima \Phi$ contains an arc in $S^{1}$.
\end{lemma}

\begin{proof}
We will prove Lemma \ref{lem:containarc} by contraposition. Suppose $\Ima\Phi$ doesn't include any arc in $S^{1}$. We aim to demonstrate that $G$ is a straight line.

Recall that $\Ima \Phi$ has at most two connected components; under our premise, each must be a single point. We consider two cases:

\textbf{Case 1}: If $\Ima \Phi$ consists of a single point, then all directions from $p^*$ to any point in $G$ are constant, implying $G$ is a straight line.

\textbf{Case 2}: If $\Ima \Phi$ consists of two points, these must be antipodal due to the invariance of $\Ima \Phi$ under $\rho_\theta$. This implies all points in $G$ are collinear, so $G$ is again a straight line.

Consequently, in either case, if $\Ima \Phi$ does not include any arc in $S^{1}$, it follows that $G$ must be a straight line, as we wanted.

\end{proof}

With the aid of Lemma \ref{lem:containarc}, we now proceed to the proof of Proposition \ref{prop:rigidrotation}.



\begin{proof}[Proof of Proposition \ref{prop:rigidrotation}]

    In fact, if $G$ is already a straight line, then Proposition \ref{prop:rigidrotation} trivially holds. If not, we can further assume that the graph $G$ of a continuous function is not a straight line. Applying  Lemma \ref{lem:containarc}, then $\Ima \Phi$ must then contain an arc.

    Choose such an arc $J \subset \Ima \Phi$ on $S^1$. According to \eqref{equ:forward}, we have $\rho_\theta^i(J) \subset \Ima \Phi$ for any $i \in \mathbb{N}_0$. Consequently,
\begin{equation}\label{equ:finiteiteration}
    \bigcup_{i=0}^{n} \rho_\theta^i(J) \subset \Ima \Phi \quad \text{for any } n \in \mathbb{N}_0,
\end{equation}
and
\begin{equation}\label{equ:infiniteiteration}
    \bigcup_{i=0}^{\infty} \rho_\theta^i(J) \subset \Ima \Phi.
\end{equation}

    We now divide the remainder of the proof into two cases: when $\theta/2\pi$ is irrational and when $\theta/2\pi$ is rational.

    \textbf{Case 1}: Suppose $\theta/2\pi \in \mathbb{Q}^c$. By Lemma \ref{cor:fullcircle}, $\bigcup_{m=1}^{\infty} \rho_\theta^m(J)$ covers the entire circle $S^1$. Together with \eqref{equ:infiniteiteration}, this implies $\Ima \Phi = S^1$. However, this contradicts to the fact that $\mathcal{N}$ and $\mathcal{S}$ are not in $\Ima \Phi$. Hence, Case $1$ is impossible.

    \textbf{Case 2}: Now suppose $\theta/2\pi \in \mathbb{Q}$. We can write $\theta/2\pi = m/n$ for some $m \in \mathbb{Z}$, $n \in \mathbb{N}$, with $\gcd(m, n) = 1$.

For $n \geq 3$, denote by $\mathcal{N}_i = \rho_{-\theta}^i(\mathcal{N})$, $\mathcal{S}_i = \rho_{-\theta}^i(\mathcal{S})$, and $J_i = \rho_\theta^i(J)$ for $i = 0, \dots, n-1$. In this case, $S^1$ is partitioned into $2n$ segments by the $2n$ points $\mathcal{N}_0, \dots, \mathcal{N}_{n-1}, \mathcal{S}_0, \dots, \mathcal{S}_{n-1}$ if $n$ is odd, and into $n$ segments if $n$ is even, with some points coinciding.

Regardless of $n$ being odd or even, the intervals $J_0, \dots, J_{n-1}$ fall into $n$ different segments, as illustrated in Figure \ref{pic:illustration}.

\begin{figure}[h]
    \centering
    \includegraphics[width=0.85\linewidth]{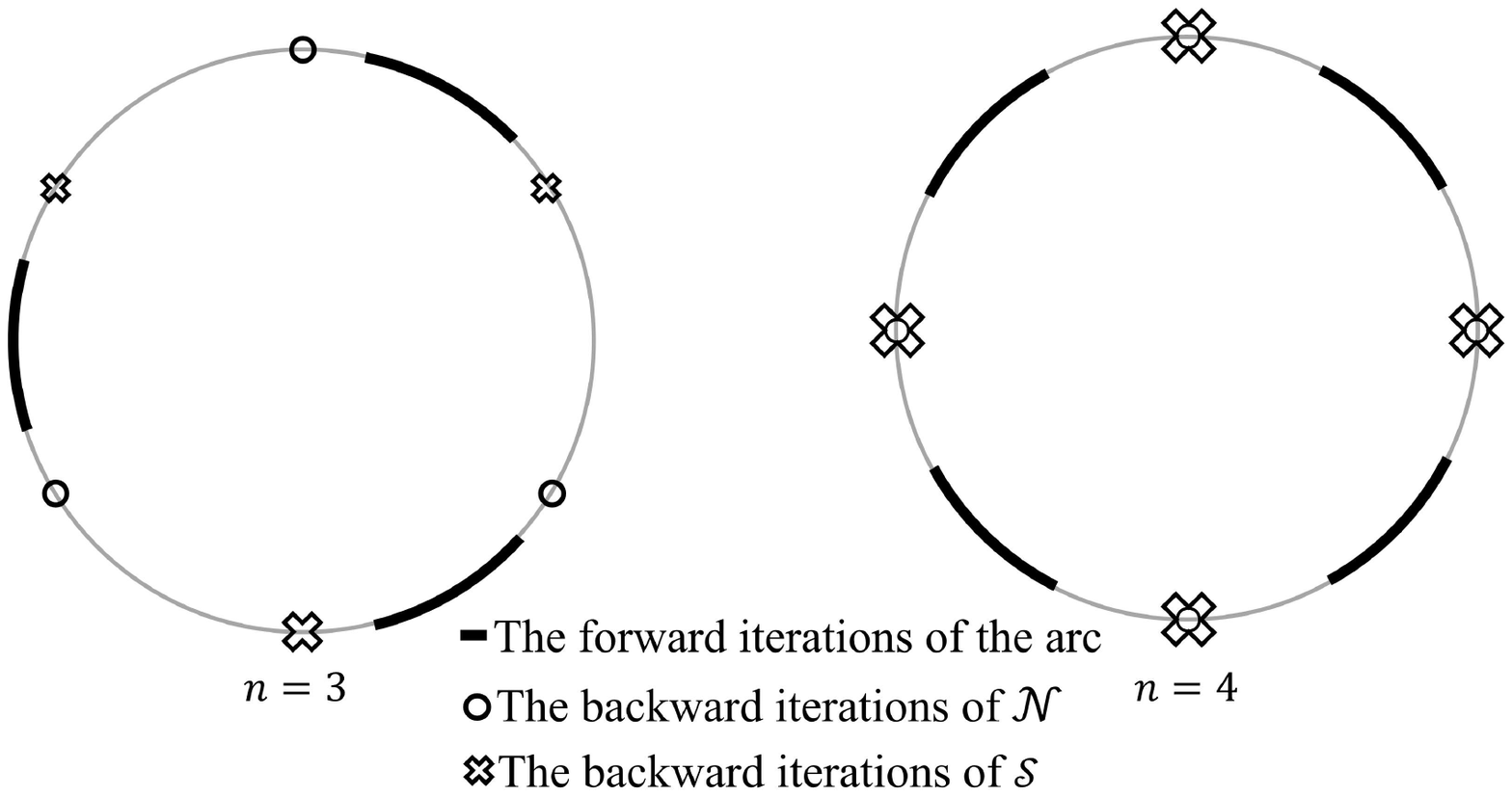}
    \caption{The iterations of $\mathcal{N},\mathcal{S}$ and $J$ on $S^1$}
    \label{pic:illustration}
\end{figure}

Observe that the complement of $\Ima \Phi$ in $S^1$ is backward invariant under the rotation $\rho_{\theta}$, i.e.,
\begin{equation}\label{equ:backward}
\rho_{-\theta}((\Ima \Phi)^c) \subset (\Ima \Phi)^c.
\end{equation}

Combining \eqref{completementofimage} with forward and backward invariance \eqref{equ:forward} and \eqref{equ:backward} imply that the points $\mathcal{N}_i$ and $\mathcal{S}_i$, for $i = 0, \dots, n-1$, are not contained in $\Ima \Phi$.

This fact together with \eqref{equ:finiteiteration} further implies that the number of connected components of $\Ima \Phi$ exceeds two, which is obviously a contradiction. Therefore, we conclude that $n \leq 2$.

When $n = 1$, we have $\theta = 2m\pi$, which means $\rho_\theta = \begin{pmatrix}1 & 0 \\ 0 & 1\end{pmatrix}$.

When $n = 2$, we have $\theta = m\pi$, which means $\rho_\theta = \begin{pmatrix}1 & 0 \\ 0 & 1\end{pmatrix}$ or $\begin{pmatrix}-1 & 0 \\ 0 & -1\end{pmatrix}$.


\end{proof}

\subsection{Lipschitz Continuity}\label{sec:vitalicoverargument}


This subsection aims to demonstrate that the underlying function of a self-similar graph is Lipschitz, as stated in Proposition \ref{prop:Vitalicovering}. To achieve this, we initially frame $G$ within a minimal rectangle $R$. We then leverage the self-similarity property of $G$ to generate a set of smaller, similar rectangles, each with widths less than a specified threshold $\delta$. By projecting these rectangles onto the x-axis, we identify a minimal cover for a given closed interval. The function's oscillation across the projected rectangles is governed by the rectangles' heights, ensuring the Lipschitz continuity of the underlying function.

\begin{prop}\label{prop:Vitalicovering}
Suppose $f:[0,1]\to \R$ is continuous. If the graph $G$ of $f$ on $[0,1]$ is self-similar, then $f$ is a Lipschitz function.
\end{prop}

Before we proceed with the proof, let's introduce some notations. Given a bounded closed interval $I$ and a function $f:I \to \R$, the \emph{oscillation} of $f$ on $I$, is denoted by
\begin{equation}\label{equ:defofoscillation}
    \omega_f(I)\coloneqq \sup\limits_{x,y\in I }|f(x)-f(y)|.
\end{equation}

We say the graph of a continuous function $f$ on a compact interval $I$ is \emph{framed} by a rectangle $R=[x_1,x_2]\times[y_1,y_2]$, if $I=[x_1,x_2]$ and $y_1\coloneqq\min\limits_{x\in I}f(x),~y_2\coloneqq\max\limits_{x\in I}f(x)$. By continuity of $f$ and compactness of $I$, the height of $R$ equals to the oscillation of $f$ on $I$.




\begin{proof}[Proof of Proposition \ref{prop:Vitalicovering}]
Since $f$ is continuous on $[0,1]$, $\omega_f < \infty$.

We aim to show that
\begin{claim}\label{claim:lip}
For every $\delta>0$,
\begin{equation}\label{equ:lip}
|f(x) - f(y)| \leq 4\omega_f\cdot|x-y|,
\end{equation}
for every $x, y \in [0,1]$ with $|x - y| = \delta$.
\end{claim}
Since $\delta,x,y$ are arbitrarily chosen, \eqref{equ:lip} directly yields that $f$ is a $4\omega_{f}$-Lipschitz function, and thus will complete the proof of Proposition \ref{prop:Vitalicovering}.

In the rest of the proof, we will complete the proof of Claim \ref{claim:lip}.
By exchanging $x$ and $y$ if necessary, we can assume that $x<y$. Meanwhile, by the hypothesis, the graph $G$ of $f$ is self-similar, so recall the notations in Section \ref{sec:preliminaries}, we have $G=\bigcup_{i=1}^{k}S_{i}(G)$ for some constant $k\in\mathbb{N}$. Each $S_{i}(v)=r_{i}\rho_{\theta_{i}}(v)+b_{i}$, with $r_{i}\in(0,1), b_{i}\in\mathbb{R}^{2}$ and $\rho_{\theta_{i}}$ is a rotation matrix on $\mathbb{R}^{2}$, for every $i\in[k]$.

Let $R$ be the rectangle framing $f$ on $[0,1]$. For any $n\in \N$ and $\alpha\in [k]^n$. Denote by
\begin{equation*}\label{def:definitionofprojection}
    I_\alpha\coloneqq \pi\circ S_\alpha (R),
\end{equation*}
where $\pi$ is the projection given by $\pi(x_1,x_2):=x_1$.

For each $i\in[k]$, due to $S_i(G)\subset G$, it follows from Proposition \ref{prop:rigidrotation} that $\rho_{\theta_i}=\begin{pmatrix}1&0\\0&1 \end{pmatrix}$ or $\begin{pmatrix}-1&0\\0&-1 \end{pmatrix}$ . Moreover, $S_i(G)$ is the graph of $f$ on $I_i$ and the rectangle $S_i(R)$ frames the graph of $f$ on $I_i$, see Figure \ref{fig:frame}.
\begin{figure}[h]
    \centering
    \includegraphics[width=0.8\linewidth]{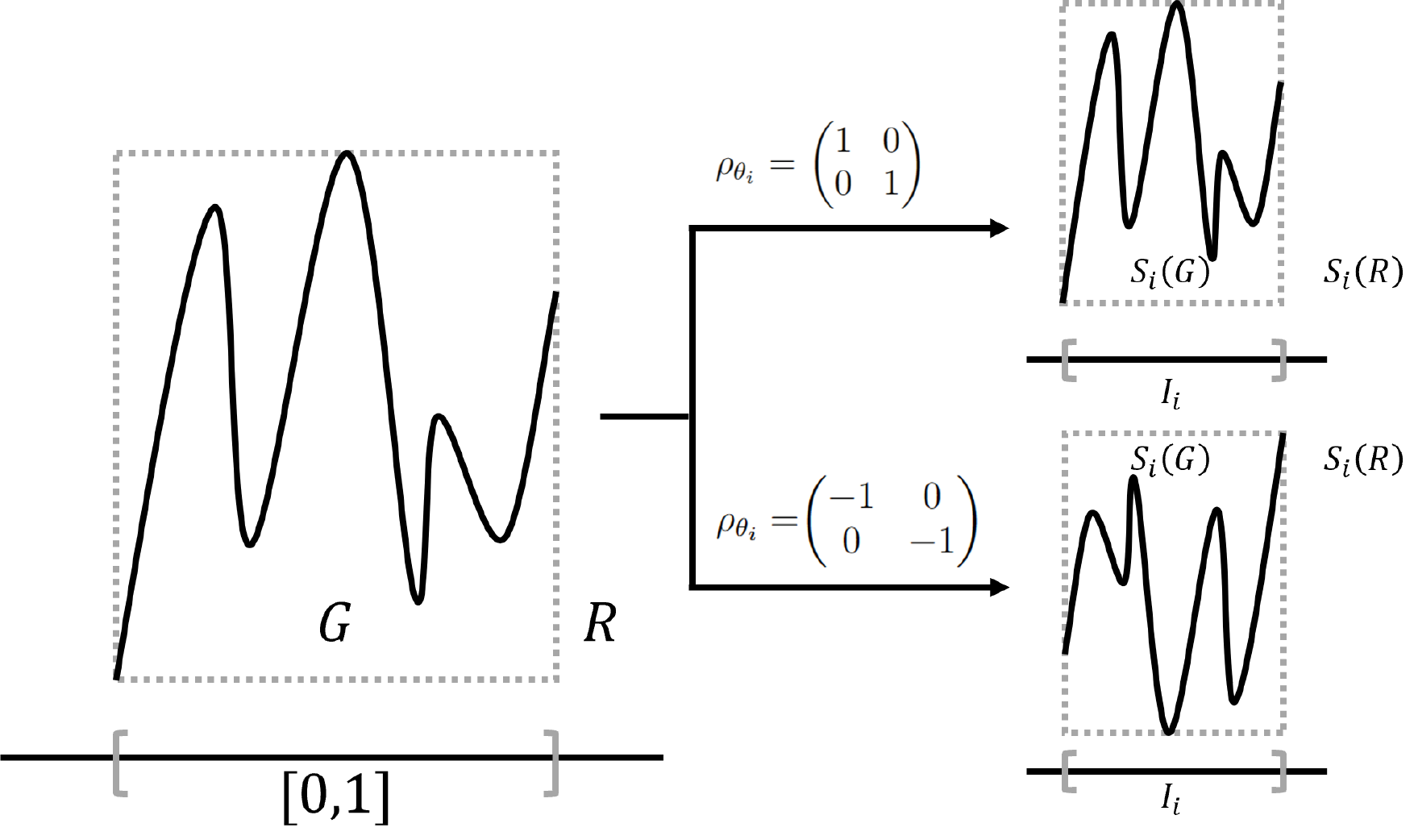}
    \caption{$S_i(R)$ frames $f$ on $I_i$}
    \label{fig:frame}
\end{figure}

Obviously, $S_i(R)$ is a rectangle with width $|I_i|=r_i$ and height $\omega_f(I_i)=r_i\cdot \omega_f$. Hence,
$$
\omega_f(I_i)=|I_i|\cdot \omega_f, \text{ for }i\in[k].
$$

For every $n\in \mathbb{N}$ and $\alpha\in [k]^n$, we inductively repeat the above process for $n$-steps on each $I_{\alpha}$ and obtain that
\begin{enumerate}
    \item $S_\alpha(G)$ is the graph of $f$ on $I_\alpha$;
    \item $S_\alpha(R)$ frames the graph of $f$ on $I_\alpha$, where $S_\alpha(R)$ is a rectangle with width $|I_\alpha|=r_\alpha$ and height
    \begin{equation}\label{equ:fromintervaltooscillation}
    \omega_f(I_\alpha)=r_\alpha \cdot \omega_f.
\end{equation}
\end{enumerate}

On the other hand, since $r_1,\dots, r_k\in (0,1)$, we have $r_{\max}\coloneqq\max\{r_1,\dots, r_k\}\in(0,1)$. Thus, for every $\delta>0$ given in Claim \ref{claim:lip}, we can choose a $n_{0}:=n_{0}(\delta)\in \N$ with $r_{\max}^{n_0}\leq \delta$. This implies that for any $\alpha\in[k]^{n_0}$, we have
\begin{equation}\label{equ:lengthofeachinterval}
    |I_\alpha|=r_\alpha \leq r_{\max}^{n_0}\leq \delta.
\end{equation}

The self-similarity of $G$ indicates that
$$G=\bigcup\limits_{i=1}^{k}S_i(G)=\bigcup\limits_{\alpha\in[k]^{n_0}}S_\alpha(G).$$ Hence $G$ is covered by $\{S_\alpha(R)\}_{\alpha\in[k]^{n_0}}$. Consequently, the interval $[0,1]$ is covered by a collection of compact intervals $\{I_\alpha\}_{\alpha\in[k]^{n_0}}$.

Due to the finiteness of the indices set $[k]^{n_0}$, we can select a finite indices subset $\Lambda \subset [k]^{n_0}$ that satisfies the following two conditions:

\begin{enumerate}
    \item $\{I_\alpha\}_{\alpha\in\Lambda}$ forms a cover for $[x,y]$.
    \item $\{I_\alpha\}_{\alpha\in\Lambda}$ is minimal, in the sense that if any $\alpha_0\in \Lambda$ is removed, then $\{I_\alpha\}_{\alpha\in\Lambda\backslash \{\alpha_0\}}$ no longer forms a cover for $[x,y]$.
\end{enumerate}


We arrange the set $\{I_{\alpha}\}_{\alpha\in\Lambda}$ into an ordered sequence ${I_1, \dots, I_m}$ based on the left endpoints, from left to right. According to Condition (2) in our construction, for every indices $j\in[m-2]$, the intervals $I_j$ and $I_{j+2}$ are disjoint, i.e., $I_j \cap I_{j+2} = \emptyset$. To see this, assume for the sake of contradiction that there is an index $j_0$ in $[m-2]$ such that the intersection $I_{j_0} \cap I_{j_0+2}$ is non-empty. Given our ordering, this would suggest that $I_{j_0+1}$ is entirely contained within the union of $I_{j_0}$ and $I_{j_0+2}$. This implies that the set ${I_1, \dots, I_m}$ excluding ${I_{j_0+1}}$ would still provide a cover for the interval $[x,y]$, contradicting Condition (2) as defined in $\Lambda$.


Notice that the intervals $I_j$, where $j\in[m]\backslash\{1,m\}$ and $j$ is odd, are mutually disjoint and their union is contained in $[x,y]$. Therefore,

$$\sum\limits_{\substack{j\in[m]\backslash\{1,m\}\\j \text{ odd}}} |I_j| \leq |[x,y]|=\delta.$$

Similarly, we also have

$$\sum\limits_{\substack{j\in[m]\backslash\{1,m\}\\j \text{ even}}} |I_j| \leq \delta.$$

On the other hand, due to \eqref{equ:lengthofeachinterval}, $|I_1|,|I_m|<\delta$. Hence the total length
\begin{equation}\label{equ:totallength}
    \sum\limits_{j=1}^{m} |I_j| \leq 4\delta.
\end{equation}

Finally, we arbitrarily choose points $x_i\in I_i\cap I_{i+1},i\in[m-1]$. Consider,
$$\begin{aligned}
    |f(x)-f(y)|&\leq |f(x)-f(x_1)|+\sum\limits_{i=1}^{m-2}|f(x_i)-f(x_{i+1})|\\
    &+|f(x_{m})-f(y)|&\text{(by triangle inequality)}\\
    &\leq \omega_f(I_1)+\sum\limits_{i=1}^{m-2}\omega_f(I_{i+1})+\omega_f(I_m)
    &\text{(by \eqref{equ:defofoscillation})}\\
    &=|I_1|\cdot \omega_f+\sum\limits_{i=1}^{m-2}|I_i|\cdot \omega_f+|I_m|\cdot \omega_f
    &\text{(by \eqref{equ:fromintervaltooscillation})}\\
    &\leq 4\delta \omega_f.&\text{(by \eqref{equ:totallength})}
\end{aligned}$$
Therefore, we complete the proof of Claim \ref{claim:lip}, thus obtaining    Proposition \ref{prop:Vitalicovering}.
\end{proof}


\subsection{Proof of Theorem \ref{thm:maintheormrigidity}}\label{sec:bandtkravchenko'sresult}

The aim of this subsection is to show that the underlying function of a self-similar graph is affine.

To prove this, we demonstrate that any interval $[a,b]$ within $[0,1]$ contains a subinterval $[s,t]$ of comparable length to $[a,b]$ satisfying the condition $\frac{f(t)-f(s)}{t-s}=f(1)-f(0)$ (see Proposition \ref{lem:attacher}). By repeatedly applying Proposition \ref{lem:attacher}, we can construct a Cantor-like subset within $[a,b]$. This construction, combined with the fact that $f$ satisfies Lipschitz continuity (see Proposition \ref{prop:Vitalicovering}), allows us to deduce that $f$ is affine.

\begin{prop}\label{lem:attacher}
    Let $G$ be the graph of a continuous function $f$ on $[0,1]$. If $G$ is self-similar, then there exists $c\in (0,1/2)$ such that for any closed interval $[a,b]\subset [0,1]$, there exists a closed subinterval $[s,t]\subset [a,b]$ satisfying
    \begin{enumerate}
        \item $\frac{a+b}{2}\in [s,t]$;
        \item $c(b-a)\leq t-s\leq \frac{1}{2} (b-a)$;
        \item $\frac{f(t)-f(s)}{t-s}=f(1)-f(0):=\lambda$.
    \end{enumerate}
\end{prop}

\begin{proof}
   By hypothesis $G$ is self-similar, so we adhere to the notations $k,S_{i},r_{i}$ in the proof of Proposition \ref{prop:Vitalicovering}. Let $r_{\min}\coloneqq \min\{r_1,\dots,r_k\}>0$, and take
   \begin{equation*}
   c\coloneqq \frac{r_{\min}}{2}>0.
   \end{equation*}
   Let's fix an interval $[a,b]\subset[0,1]$. Due to the self-similarity of $G$ again, for any integer $n\in\N$, there exists $\alpha\in [k]^n$ such that $I_\alpha$ contains the middle point $\frac{a+b}{2}$.
    Let $n^{*}\in\mathbb{N}$ be the smallest integer, such that $|I_\alpha|\leq \frac{b-a}{2}$ .
    We will show that
    \begin{equation*}
    I_{\alpha}:=[s,t]
    \end{equation*}
    is the desired subinterval satisfying three properties mentioned above.

Let us verify the properties item by item. Property (1) and the second inequity in Property (2) are immediate from the construction of $I_{\alpha}$.

For the first inequity in Property (2), we write $\alpha=(\alpha_1,\dots,\alpha_{n^{*}})\in [k]^{n^{*}}$ and denote by $\alpha'\coloneqq (\alpha_1,\dots,\alpha_{n^{*}-1})\in [k]^{n^{*}-1}$. Due to  the smallest integer property of $n^{*}$, it follows that $|I_{\alpha'}|>\frac{b-a}{2}$. Thus, $|I_\alpha|=r_{\alpha_{n^{*}}}\cdot |I_{\alpha'}|>r_{\alpha_{n^{*}}}\cdot\frac{b-a}{2}$, which yields $$
|I_\alpha|\geq r_{\min}\cdot\frac{b-a}{2}=c(b-a).
$$

For Property (3), let us denote $\overrightarrow{x} \coloneqq (x, f(x)) \in G$ for $x \in [0,1]$. Let us first prove the following claim.

\begin{claim}\label{claimfinal}
    For every  $i\in[k]$, if $\{\overrightarrow{a},\overrightarrow{b}\}=\{S_i(\overrightarrow{c}),S_i(\overrightarrow{d})\}$, then
    $$
    \frac{f(b)-f(a)}{b-a}=\frac{f(d)-f(c)}{d-c}.
    $$
\end{claim}


Since the rotation associated with any similitude $S_i$ must be either the identity or a reflection, $(x, f(x))$ is mapped by $S_i$ to either $(r_i x + b_i, r_i f(x) + b_i)$ or $(-r_i x + b_i, -r_i f(x) + b_i)$. In both cases, the slope remains invariant. This concludes the proof of Claim \ref{claimfinal}.

Finally, since $\{\overrightarrow{s}, \overrightarrow{t}\} = \{S_\alpha(\overrightarrow{0}), S_\alpha(\overrightarrow{1})\}$, we apply Claim \ref{claimfinal} $n^*$ times and obtain that
$$
\frac{f(t) - f(s)}{t - s} = \frac{f(1) - f(0)}{1 - 0} = \lambda,
$$
as desired. This concludes our proof of Proposition \ref{lem:attacher}.
\end{proof}




Finally, we are ready to prove Theorem \ref{thm:maintheormrigidity}.
\begin{proof}[Proof of Theorem \ref{thm:maintheormrigidity}]
    $``\Rightarrow":$ To show $f$ is affine, it suffices to show that for any interval $[a,b]\subset [0,1]$, we have $\frac{f(b)-f(a)}{b-a}=\lambda$.

Given an arbitrary interval $[a,b]$, we construct a Cantor-like subset of $[a,b]$ inductively in the following way.

For initial step, there exists a subinterval of $[a,b]$, say $[a_1^1,b_1^1]$, satisfying three properties in Proposition \ref{lem:attacher}. We define $C_1\coloneqq[a,b]\backslash(a_1,b_1)$. This set is non-empty due to the first two properties in Proposition \ref{lem:attacher}.

Suppose we have defined $C_n$. Notice there are $2^n$ intervals in $C_n$. For stage $n+1$, we apply Proposition \ref{lem:attacher} on each of these $2^n$ intervals, obtaining subintervals $[a^n_1,b^n_1],\dots, [a_{2^n}^n,b_{2^n}^n]\subset C_n$. We then define $C_{n+1}\coloneqq C_n\backslash\bigcup\limits_{i=1}^{2^n}(a_i^n,b_i^n)$. This set is also non-empty by the first two properties in Proposition \ref{lem:attacher}.


\begin{figure}[h]
    \centering
    \includegraphics[width=0.95\linewidth]{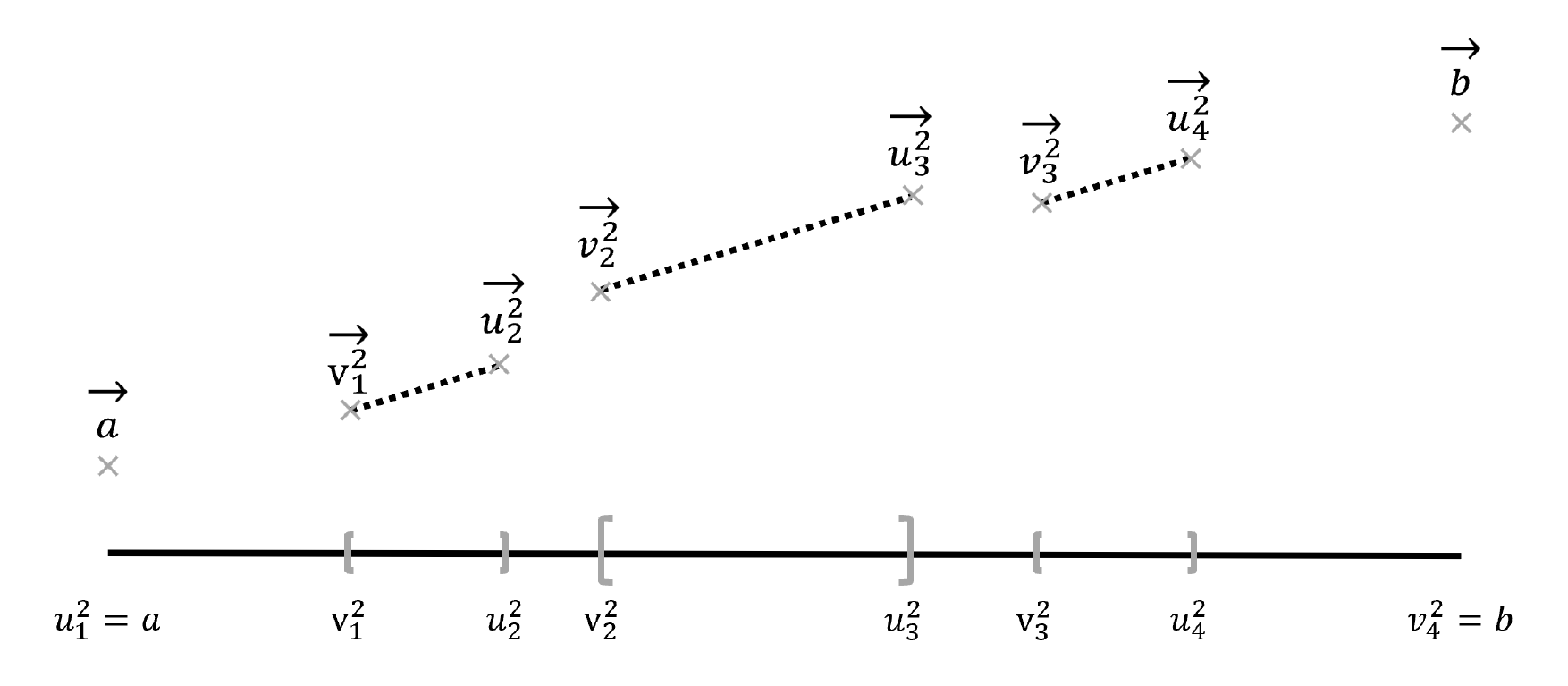}
    \caption{Illustration of the set $C_2$}
    \label{fig:stage2}
\end{figure}

Each $C_n$ contains $2^n$ intervals. We denote the right and left endpoints of the $i$-th interval as $u_i^n$ and $v_i^n$ respectively. Thus, the intervals in $C_n$ can be sequentially represented as $[u^n_1, v^n_1], [u^n_2, v^n_2], \dots, [u^n_{2^n}, v^n_{2^n}]$, as depicted in Figure \ref{fig:stage2}.

Next, we estimate the total length of $C_n$, $|C_n|=\sum\limits_{i=1}^{2^n}(v_i^n-u_i^n)$. By invoking Property (2) in Lemma \ref{lem:attacher}, we deduce that $|C_{n+1}|\leq (1-c)|C_n|$. Combined this with the fact that $|C_1|\leq (1-c)(b-a)$, it follows that
\begin{equation}\label{equ:totallengthatstagen}
 |C_n|\leq (1-c)^n(b-a).
\end{equation}

Meanwhile, by Proposition \ref{prop:Vitalicovering}, there is a constant $L=4\omega_{f}>0$, such that for any $n\in \N$, we have
\begin{equation}\label{equ:lipschitzinterval}
    |f(v_i^n)-f(u_i^n)|\leq L|v_i^n-u_i^n|\text{, for all }i\in[2^n].
\end{equation}

On the other hand, by Property (3) in Proposition \ref{lem:attacher}, it follows that
 \begin{equation}\label{equ:oscillationoneachinterval}
     f(u^n_{i+1})-f(v_i^n)=\lambda(u_{i+1}^n-v_i^n)\text{, for all }i\in[2^n-1].
 \end{equation}

Based on all the estimates, for every $n\in\N$, we have
 \begin{align*}
    &|f(b)-f(a)-\lambda(b-a)|\\
    =&|\sum\limits_{i=1}^{2^n}(f(v_i^n)-f(u_i^n))+\sum\limits_{i=1}^{2^n-1}(f(u_{i+1}^n)-f(v_i^n))-\lambda(b-a)|\\
    =&|\sum\limits_{i=1}^{2^n}(f(v_i^n)-f(u_i^n))+\sum\limits_{i=1}^{2^n-1}\lambda(u_{i+1}^n-v_i^n)-\lambda(b-a)|
    &\enskip&\text{(by \eqref{equ:oscillationoneachinterval})}\\
    \leq& \sum\limits_{i=1}^{2^n}|f(v_i^n)-f(u_i^n)|+|\lambda(b-a-|C_n|)-\lambda(b-a)|\\
    \leq& (L+|\lambda|)|C_n|
    &\enskip&\text{(by \eqref{equ:lipschitzinterval})}\\
    \leq& (L+|\lambda|)(1-c)^n(b-a).
    &\enskip&\text{(by \eqref{equ:totallengthatstagen})}
\end{align*}

Since $c= \frac{r_{\min}}{2}<1$ and $n$ is arbitrarily chosen, we conclude that $f(b)-f(a)-\lambda(b-a)=0$ for any interval $[a,b]\subset [0,1]$. Therefore, $f$ is an affine function.



    $``\Leftarrow":$ Conversely, suppose $f$ is an affine function. Similitudes $S_1(v) = \frac{1}{2}v + (0, \frac{f(0)}{2})$ and $S_2(v) = \frac{1}{2}v + (\frac{1}{2}, \frac{f(1)}{2})$ satisfies $G = S_1(G) \cup S_2(G)$. Hence, $G$ is self-similar.
\end{proof}

\end{document}